\documentclass[11pt,reqno]{amsart}
\usepackage{mathtools}
\usepackage{amssymb,amsmath}
\usepackage[T1]{fontenc}
\usepackage{wrapfig}
\usepackage{amscd}
\usepackage{amsrefs}
\usepackage{graphicx}
\usepackage{hyperref}

\def\H{\mathbf{H}}

\def\text#1{\hbox{#1}}

\def\R{{\mathbb R}}
\def\N{{\mathbb N}}

\def\H{{\mathbf H}}

\def\O{{\mathcal O}}

\def\Irr{{\rm Irr}}

\def\Hom{{\rm Hom}}

\def\GL{{\rm GL}}

\def\N{{\rm N}}

\def\tr{{\rm tr}}

\def\G{{\mathbf{G}}}
\def\O{{\mathcal{O}_F}}
\def\cInd{{\operatorname{c-Ind}}}
\def\Irr{{\rm{Irr}^{sc}}}

    \DeclareFontFamily{U}{wncy}{}
    \DeclareFontShape{U}{wncy}{m}{n}{<->wncyr10}{}
    \DeclareSymbolFont{mcy}{U}{wncy}{m}{n}
    \DeclareMathSymbol{\Sh}{\mathord}{mcy}{"58}

\newtheorem{theorem}{Theorem}[section]

\newtheorem{lemma}[theorem]{Lemma}
\newtheorem{corollary}[theorem]{Corollary}
\newtheorem{proposition}[theorem]{Proposition}
\newtheorem{conjecture}[theorem]{Conjecture}

\date{\today}
\theoremstyle{definition}
\numberwithin{equation}{section}
\usepackage{pgfplots}

\newcommand\blfootnote[1]{%
  \begingroup
  \renewcommand\thefootnote{}\footnote{#1}%
  \addtocounter{footnote}{-1}%
  \endgroup
}

\def\N{{\mathbb N}}

\title{Uniform bounds on the Harish-Chandra characters} 
\date{}

 \author{Anna Szumowicz}

  \address{Caltech, The Division of Physics, Mathematics and Astronomy,
1200 E California Blvd, Pasadena CA 91125}
\email{anna.szumowicz@caltech.edu}
\begin{document}
\maketitle
\blfootnote{\today}
\begin{abstract}
 Let $\G$ be a connected reductive algebraic group over a $p$-adic local field $F$. In this paper we study the asymptotic behaviour of the trace characters $\theta _{\pi}$ evaluated at a regular element $\gamma $ of $\G(F)$ as $\pi$ varies among supercuspidal representations of $\G(F)$. Kim, Shin and Templier \cite{KST2016} conjectured that $\frac{\theta_{\pi}(\gamma)}{\deg(\pi)}$ tends to $0$ when $\pi$ runs over irreducible supercuspidal representations of $\G(F)$ with unitary central character and the formal degree of $\pi$ tends to infinity. For $\G$ semisimple we prove that the trace character is uniformly bounded on $\gamma$ under the assumption, which is expected to hold true for every $\G (F)$, that all irreducible supercuspidal representations of $\G(F)$ are compactly induced from an open compact modulo center subgroup. Moreover, we give an explicit upper bound in the case of $\gamma $ ellitpic. \end{abstract} 

\section{Introduction}
Let $\G$ be a connected reductive algebraic group over a $p$-adic local field $F$. Denote $G:=\G(F)$. Let $\pi$ be an admissible representation of $G$. Denote by $\mathcal{C}^{\infty}_{c}(G)$ the set of smooth compactly supported functions on $G$. The trace character $\theta _{\pi}$ is the conjugation invariant distribution on $\mathcal{C}_{c}^{\infty}(G)$ defined as
\begin{equation*} 
\langle \theta_{\pi},f\rangle :=\tr\, \pi (f)
\end{equation*}
for $f\in \mathcal{C}_{c}^{\infty}(G)$. Harish-Chandra showed that $\theta _{\pi}$ is represented by a locally constant function on the set of regular semisimple elements which is locally $L^{1}$ on $G$. We will still denote this function by $\theta _{\pi}$.  

If $\pi$ is a discrete series representation, for example is irreducible unitary supercuspidal, one can define the formal degree of $\pi$ denoted by $\deg (\pi)$. It plays a similar role as the dimension of representations of compact groups. 

Denote by $\Irr(\G(F))$ the set of equivalence classes of irreducible supercuspidal representations of $G$. Kim, Shin and Templier proposed the following conjecture.
\begin{conjecture}[{\cite[Conjecture 4.1.]{KST2016}}]\label{conjectureKST}
 Fix a regular semisimple element $\gamma \in G$. Then,
\begin{equation*}    
\lim _{\substack{\pi \in \Irr (G) \\ \deg(\pi )\to \infty}}
\frac{\theta _{\pi}(\gamma)}{\deg (\pi)}= 0 
\end{equation*} 
where $\pi \in \Irr(\G (F))$ varies over representations whose central character is unitary.
\end{conjecture} 

Kim, Shin and Templier \cite{KST2016} proved Conjecture \ref{conjectureKST} in case of 
$\gamma $ topologically unipotent, $\pi $ constructed by Yu \cite{Yu2001} and under certain assumptions on $G$. In \cite{KST2} they proved Conjecture \ref{conjectureKST} for $\gamma$ elliptic and under certain assumption on tori in $\G$ which is satisfied for example when every $F$-torus in $\G$ is tame and the characteristic of the residue field of $F$ does not divide the order of the Weyl group of $\G$. An analogue of this conjecture for compact semisimple Lie groups can be easily verified using the Weyl character formula. See also \cite{Witte} for partial results in the case of the general linear group.      

As described in \cite{KST2}, bounds on the trace characters are motivated by their relation with limit multiplicity formulas, Sato-Tate equidistribution and bounds towards Ramanujan.

\subsection{Uniform bounds}
In this paper we prove Conjecture \ref{conjectureKST} for arbitrary semisimple connected $p$-adic groups for irreducible supercuspidal representations of the form $\cInd _{J}^{G}\Lambda$ where $J$ is an open compact modulo center subgroup of $G$. The condition is expected to hold true for any irreducible supercuspidal representation. It has been proven that every irreducible supercuspidal representation is a compact induction from an open compact modulo center subgroup in the following cases: tamely ramified $p$-adic group $\G$ when $p$ does not divide the order of the Weyl group of $\G$ (Yu's construction \cite{Yu2001} and the exhaustion results by Kim \cite{Kim2007} and Fintzen \cite{Fintzen}), for all supercuspidal representations of $\GL_{n}(F)$ (Bushnell-Kutzko construction \cite{BK}) and for the classical groups (Stevens \cite{Stevens}). 
\newline Kim, Shin and Templier in their work \cite{KST2016}, \cite{KST2} towards Conjecture \ref{conjectureKST} rely on Yu's construction of supercuspidal representations. Our methods are different and avoid Yu's construction, instead using the Arthur's formula and bounds on irreducible characters of compact $p$-adic groups from \cite{Fraczyk}. 
We will say that a representation $\pi$ of $\G(F)$ is compactly induced if it is of the form $\pi \cong \cInd _{J}^{\G(F)}\Lambda $ for some compact modulo $Z(\G(F))$ subgroup $J$ and some irreducible representation $\Lambda $ of $J$.
\begin{theorem}\label{thmmain}
Let $\G$ be a connected semisimple group over a non-Archimedean local field $F$ of characteristic zero. 
Then, for every regular semisimple element $\gamma \in \G(F)$ there exists a constant $\kappa _{\gamma}>0$ depending only on $\gamma$ and $G$ such that for every irreducible compactly induced supercuspidal representation $\pi$ we have 
\begin{equation*} 
|\theta_{\pi}(\gamma)|\leqslant \kappa_{\gamma}. 
\end{equation*}
\end{theorem}
 In particular, it implies the conjecture.
 \begin{corollary} 
 Let $\G$ be a connected semisimple group over a local non-Archimedean field $F$ of characteristic zero. 
 Then for every regular semisimple element $\gamma \in \G(F)$ we have 
\begin{equation*} 
\lim_{\substack{\pi\in \Irr (\G(F))\\\deg(\pi)\to \infty}}\frac{\theta _{\pi}(\gamma)}{\deg (\pi )}=0
\end{equation*} 
where each $\pi $ in the limit is assumed to be compactly induced.
 \end{corollary}
 For $\gamma $ ellitpic we prove an explicit upper bound. 
 \begin{theorem} 
 Let $\G$ be a connected semisimple group over a non-Archimedean local field $F$ of characteristic zero. Let $W$ be the Weyl group of $\G$. Then, for every regular elliptic element $\gamma\in \mathbf{G} (F)$and every irreducible compactly induced supercuspidal representation $\pi$ we have 
 \begin{equation*} 
 |\theta _{\pi} (\gamma )|\leqslant | \Delta (\gamma )|_{F}^{-1} 2^{{\rm dim}\, \G - {\rm rk}\, \mathbf{G} }\, |W|. 
 \end{equation*} 
 \end{theorem} . 
\subsection{Notation}
Let $F$ be a local non-Archimedean field. Denote by $\mathcal{O}_{F}$ its ring of integers and by $k_F$ its residue field. Write $\mathfrak{p}_{F}$ for the maximal ideal of $\O $. Denote by $q$ the cardinality of $k_{F}$. Fix a uniformizer $\varpi _{F}$ in $F$. Denote by $v_{F}$ the associated additive valuation normalized as $v_{F}(x)=n$ for $x=u\varpi _{F}^{n}\in F$ with $u$ an inveritble element in $\O $. Denote by $| \cdot |_{F}$ the corresponding absolute value $|x|_{F}=q^{-v_{F}(x)}$. For a connected reductive group $\G$ we write ${\rm rk}\, \G$ for the absolute rank of $\G$
and $\Delta (\gamma)$ denotes the Weyl discriminant of a semisimple element $\gamma$. 

Fix a connected semisimple group $\G$ over a non-Archimedean local field $F$ of characteristic $0$. We denote $G:=\G (F)$. Denote by $Z$ the center of $G$.
We fix a Haar measure $\mu $ on $G$ such that $\mu (\G(\O))=1$. Write $Z(H)$ for the center of a group $H$. We denote by $|S|$ the cardinality of a set $S$. If $\gamma $ is a regular semisimple element in $\G$, then we denote by $A_{M}(\gamma)$ the maximal split component of center of the Levi subgroup $M$ of a rational parabolic subgroup in $\G$ containing $\gamma$ and minimal with this property. Denote by $A_{\G}$ the maximal split component of the center of $\G$. 
\subsection{Outline of the proof of Theorem \ref{thmmain}}
In section \ref{sectioncompactness} we define
\begin{equation*}
\Omega (K):=\{ g\in A_{M}(\gamma)\setminus G|\ \ g^{-1}\gamma g\in K\}.
\end{equation*}The goal of section \ref{sectioncompactness} is to show that the set $\Omega (K)$ is compact (Proposition \ref{propositionomegacompact}). Since $\G$ is semisimple and $\pi$ is compactly induced we can write 
\begin{equation*} 
\pi \cong \cInd _{J}^{G}\Lambda =\cInd _{K}^{G}\rho. 
\end{equation*}
for some maximal compact open subgroup $K$ of $G$ and its irreducible representation $\rho$. 
Write $\chi _{\rho}$ for the trace of $\rho$ extended to $\G$ by $0$ outside of $K$.
Using Arthur's formula \cite{Arthur1987} we show  
\begin{equation}\label{equationtrace2} 
|\theta_{\pi}(\gamma)|=\left|\int _{\Omega(K)}\chi _{\rho}(g^{-1}\gamma g)v_{M}(g)dg\right|.
\end{equation} 

Once we know that $\Omega (K)$ is compact the existence of an upper bound is deduced from (\ref{equationtrace2}), the fact that $v_{M}$ is continuous and from the work of Fr\k aczyk \cite{Fraczyk} on upper bounds of irreducible characters of compact $p$-adic groups (Proposition \ref{propositionFraczyk}). 

\subsection{Acknowledgement} The author is grateful to Anne-Marie Aubert, Miko{\l}aj Fr\k aczyk, Sug Woo Shin and Bertrand R\'{e}my for helpful discussions. 

\section{Compactness}\label{sectioncompactness}
For a maximal compact open subgroup $K$ of $G$ we define 
\begin{align*} 
\Omega (K):=\{g\in A_{M}(\gamma )\setminus G | \ \  g^{-1}\gamma g \in \ K \}.  
\end{align*} 
One of the  main ingredients of the proof of Theorem \ref{thmmain} is the following.

\begin{proposition}\label{propositionomegacompact}
For any maximal compact open subgroup $K$ of $G$ the set $\Omega (K)$ is compact.
\end{proposition}
Before the proof we establish some auxiliary results. Recall that a regular semisimple element $\gamma \in G$ is called elliptic if $\gamma$ does not belong to any proper parabolic subgroup of $G$.

\begin{lemma}\label{lemmaSisfinite} 
Let $\H$ be a reductive group over a non-Archimedean local field $F$. Denote $H:=\H (F)$. Let $\gamma$ be an elliptic element in $H$. Let $K_{H}$ be a maximal compact subgroup of $H$. Then, the set 
\begin{equation*} 
\{ h\in Z(H)\setminus H| \ \ h^{-1}\gamma h\in K_{H}\} 
\end{equation*} is compact. 
\end{lemma} 
\begin{proof}
For the sake of contradiction assume that there exists a sequence $h_{l}\in H$, $l\in\N$ which escapes to infinity and such that $h_{l}^{-1}\gamma h_{l}\in K_{H}$. Let $\mathcal{B}(\H, F)$ be the Bruhat-Tits building of $\H$. Every maximal compact subgroup of $H$ stabilizes some vertex in $\mathcal{B}(\H,F)$. Let $x$ be a point in $\mathcal{B}(\H,F)$ which is stabilized by $K_{H}$. Then,
\begin{equation*} 
\gamma h_{l} x=h_{l}x.
\end{equation*} 
Since the action of $\H$ on its building is proper, the sequence $(h_{l}x)_{l}$ escapes to infinity. Consider the Satake-Berkovich compactification $\overline{\mathcal{B}(\H,F)}$ of $\mathcal{B}(\H,F)$ (see \cite{R1}, \cite{R2},\cite{Berkovich}). Passing to a subsequence we can assume that the sequence $(h_{l}x)_{l}$ tends to a point in the boundary of $\overline{\mathcal{B}(\H,F)}$, say $y$. Since the action of $\H$ on its building is continuous, the element $\gamma $ stabilizes the point $y$. The stabilizer of a point in the boundary of $\overline{\mathcal{B}(\H,F)}$ is contained in a proper parabolic subgroup of $H$ (see e.g. \cite[Theorem 14.11.]{R1}), so $\gamma$ is contained in a proper parabolic subgroup of $H$. This leads to a contradiction with the assumption that $\gamma $ is elliptic.  
\end{proof}
 
Recall that for $\G$ connected and an $F$-torus $T\subseteq \G$ we can define a root subgroup in the following way. A subgroup $U\subseteq \G$ is called a root subgroup (with respect to $T$) if there exists an isomorphism $u\colon\G _{a}\to U $ and a root $\alpha $ of $\G$ with respect to $T$ such that  \begin{equation*} tu(x)t^{-1}=u(\alpha (t)x),
\end{equation*}
for every $x\in \G_{a}$, $t\in T$. 
If $T$ is split then the root $\alpha $ determines the root subgroup uniquely. We denote by $U_{\alpha}$ the root subgroup associated to $\alpha$. 
\begin{lemma} \label{lemmaconjugationbyn}
Let $\G$ be a split connected reductive group over a non-Archimedean local field $F$. Assume that the maximal torus $T$ containing $\gamma$ is split over $F$. Let $P$ be a minimal parabolic subgroup of $\G(F)$ containing $\gamma$.  Denote by $N$ its unipotent radical. Let $\Delta ^{+}$ be the set of positive roots of $T$ associated to $P$ with order $\preceq$. Write $\Delta ^{+}=\{ \alpha _{1},\ldots ,\alpha_{r}\}$ with $\alpha_{i}\preceq\alpha_{j}$ for every $i<j$.  Denote by $U_{\alpha}$ the root subgroup associated to a positive root $\alpha$. Then,
\begin{enumerate} 
\item $u_{\alpha}(x)^{-1}\gamma u_{\alpha}(x)= u_{\alpha }\left( (\alpha (\gamma)-1)x\right)\gamma$ for any $u_{\alpha}(x)\in U_{\alpha}$;
\item Let $m=\dim N$ and $n=u_{\alpha _{
1}}(x_{1}) \ldots u_{\alpha _{m}}(x_{m})$. Then,
\begin{equation*} 
n^{-1}\gamma n=\prod_{i=0}^{m-1}u_{\alpha _{m-i}}\left( (\alpha_{m-i}(\gamma)-1)x_{m-i}+W^{m}_{m-i}(\alpha_{1}(\gamma),\ldots ,\alpha_{m}(\gamma),x_{1},\ldots ,x_{m})\right) \gamma 
\end{equation*} 
where $W^{m}_{m-i}$ is a polynomial in $2m$ variables with coefficients in $F$, depending only on $i$, $m-i$ and $\alpha_{j}(\gamma)$, $x_{j}$ such that $\alpha_{j}\prec \alpha _{i}$.  


\end{enumerate}
\end{lemma} 
\begin{proof} 
$\mathbf{(1)}$ Follows from the formula
\begin{equation}\label{equation1} 
\gamma u_{\alpha}(x)\gamma^{-1}= u_{\alpha}(\alpha (\gamma )x).
\end{equation}
$\mathbf{(2)}$ We prove the following statement. 
\newline Let $1\leqslant r \leqslant m$ and $n_{1}=u_{\alpha_{1}}(x_{1})\ldots u_{\alpha _{r}}(x_{r})$. Then,
\begin{equation*} 
n_{1}^{-1} \gamma n_{1}=\prod _{i=0}^{m-1}u_{\alpha _{m-i}}((\alpha _{m-i}(\gamma)-1)x_{m-i}+W_{m-i}^{r}(\alpha_{1}(\gamma),\ldots , \alpha _{m}(\gamma),x_{1},\ldots  ,x_{m}))\gamma
\end{equation*}
where $W_{m-i}^{r}$ is a polynomial in $2m$ variables with coefficients in $F$, depending only on $r,i,m$ and $\alpha _{j}(\gamma),x_{j}$ such that $\alpha _{j} \prec \alpha _{i}$ and we put $x_{s}=0$ for $r+1 \leqslant s\leqslant m$. 

Proof is by induction on $r$. For $r=1$ it follows from $(1)$. Assume now that the formula is true for $r=l$ and we want to prove it for $r=l+1$. By \cite[Proposition 1.2.3.]{Adler}, for any $\alpha,\beta \in \Delta ^{+}$ we have 
\begin{equation}\label{equation2} 
[u_{\alpha}(x),u_{\beta}(y)]=\prod _{i,j>0}u_{i\alpha+j\beta}(c_{i,j;\alpha ,\beta}xy)
\end{equation}
where $[g,h]$ denotes the commutator of elements $g,h$ and $c_{i,j;\alpha,\beta}$ are integers dependent only on $\alpha$, $\beta$ and $i,j$ and $u_{i\alpha i+j\beta} $ denotes the trivial map if $i\alpha +j\beta $ is not a root. Denote $n_{1}:=u_{\alpha _{1}}(x_{1})\ldots u_{\alpha _{l}}(x_{l})$ and $n:=n_{1}u_{\alpha _{l+1}}(x_{l+1})$. Write $\alpha (\gamma):=(\alpha _{1}(\gamma),\ldots , \alpha_{m}(\gamma))$, $y:=(x_{1}, \ldots , x_{l},0,\ldots ,0)=:(y_{1},\ldots  ,y_{m})$ and $z=(x_{1},\ldots , x_{l+1},0\ldots ,0)=:(z_{1}, \ldots ,z_{m})$. By (\ref{equation1}), repeated use of (\ref{equation2}) and the inductive step,
\begin{align*} 
&n^{-1}\gamma n=u_{\alpha_{l+1}}(z_{l+1})^{-1}n_{1}^{-1}\gamma n_{1}u_{\alpha_{l+1}}(z_{l+1})=\\
&u_{\alpha_{l+1}}(-z_{l+1})\prod_{i=0}^{m-1}u_{\alpha_{m-i}}\left( (\alpha_{m-i}(\gamma)-1)y_{m-i}+W_{m-i}^{l} (\alpha (\gamma),y)\right)\gamma u_{\alpha_{l+1}}(z_{l+1})=\\
&u_{\alpha_{l+1}}(-z_{l+1})\prod_{i=0}^{m-1}u_{\alpha_{m-i}}\left( (\alpha_{m-i}(\gamma)-1)y_{m-i}+W_{m-i}^{l}(\alpha (\gamma),y)\right) u_{\alpha_{l+1}}(\alpha_{l+1}(\gamma)z_{l+1})\gamma = \\
&=\prod_{i=0}^{m-1}u_{\alpha_{m-i}}\left( (\alpha _{m-i}(\gamma)-1)z_{m-i}+W_{m-i}^{l+1}(\alpha (\gamma) , z)\right
)\gamma. 
\end{align*}
\end{proof}
\begin{lemma} \label{lemmaNgammacompact}
Let $\gamma$ be a non-elliptic element. Let $Q$ be a minimal proper parabolic containing $\gamma$. Denote by $N_{Q}$ the unipotent radical of $Q$ and by $L$ the Levi subgroup of $Q$ containing $\gamma$. Then, for any compact subset $\Omega $ of $L$ there exists a compact subset $E_{\Omega}\subseteq N_{Q}$ such that $x^{-1}l^{-1}\gamma l x \not\in K$ for every $x\in N_{Q}\setminus E_{\Omega}$ and every $l\in \Omega$.
\end{lemma} 
\begin{proof} 
First suppose that the maximal split torus $T$ containing $\gamma$ is split over $F$. For the sake of contradiction assume that there exist a compact subset $\Omega $ of $L$ and a sequence $(n_{i})_{i}\subseteq N_{Q}$ escaping to infinity such that $n_{i}^{-1}l_{i}^{-1}\gamma l_{i}n_{i}\in K$ for some $(l_{i})_{i}\subseteq  \Omega$. Write $n_{i}=u_{\alpha_{1}}\left(x_{1}^{(i)}\right)\ldots u_{\alpha _{m}}\left(x_{m}^{(i)}\right)$ where $\alpha _{1}, \ldots , \alpha _{m}$ are as in Lemma \ref{lemmaconjugationbyn} with $P=Q$ and $m=\dim N_{Q}$. Since the map
\begin{equation*} 
F^{m}\ni (x_{1},\ldots ,x_{m})\mapsto u_{\alpha _{1}}(x_{1})\ldots u_{\alpha_{m}}(x_{m})\in N_{Q}
\end{equation*} 
is a homeomorphism, there exists $1\leqslant j\leqslant m$ such that the sequence $\left(x_{j}^{(i)}\right)$ is not bounded. Since $\gamma $ is regular, $\alpha _{i}(\gamma )\neq 1$ for any $\alpha _{i} \in \Delta ^{+}$. For any $l_{i}$ the element $l_{i}^{-1}\gamma l_{i}$ is non-elliptic. By Lemma \ref{lemmaconjugationbyn}, the sequence $n_{i}^{-1}l_{i}^{-1}\gamma l_{i}n_{i}$ is not bounded which leads to a contradiction.
 
If $T$ is not split over $F$, then choose a finite extension $E/F$ such that $T$ is split over $E$ and repeat the above argument for $\gamma \in Q(E)$.   

\end{proof}

\begin{proof}[Proof of Proposition \ref{propositionomegacompact}]
We consider $\gamma $ up to $G$-conjugation so without loss of generality we can assume that $\gamma \in K$.
If $\gamma $ is elliptic then $A_{M}(\gamma)$ is trivial and Proposition \ref{propositionomegacompact} follows from Lemma \ref{lemmaSisfinite}. Assume now that $\gamma $ is non-elliptic. 
Then, there exists a proper parabolic subgroup of $G$ containing $\gamma$. Take $P$ to be a parabolic minimal with this property. Denote by $L$ its Levi subgroup and by $N$ its unipotent radical. 
Denote by $p\colon P\to P/N$ the canonical projection. For any subgroup $S$ of $P$ denote by $\overline{S}$ its image in $P/N$. 
 By Iwasawa decomposition,
\begin{align} 
\Omega (K)&=\{ g\in A_{M}(\gamma)\setminus G| \ \ g^{-1}\gamma g \in K\} \\
&=\{ lnk| \ \ l \in A_{M}(\gamma )\setminus L, n\in N, k\in K\ \ \textrm{and} \ \ k^{-1}n^{-1}l^{-1}\gamma lnk \in K\} \\
&\label{equationomega1}=\{ lnk| \ \ l \in A_{M}(\gamma) \setminus L, n\in N, k\in K\ \ \textrm{and}\ \  n^{-1}l^{-1} \gamma ln\in K\}.
\end{align}
Note that if $n^{-1}l^{-1}\gamma ln\in K$ then $l^{-l}\gamma l$ belongs to a compact subgroup $p^{-1}(\overline{K\cap P})$. By Lemma \ref{lemmaSisfinite} applied to $L$, the set
\begin{equation*} 
\Omega(K,P):=\{ l\in L| \ \ l^{-1}\gamma l\in p^{-1}(\overline{K\cap P})\}
\end{equation*} is compact. By (\ref{equationomega1}),
\begin{align*} 
\Omega (K)&=\{lnk|\ \ l\in A_{M}(\gamma)\setminus L, n\in N, k\in K,\ \ n^{-1}l^{-1}\gamma ln\in K\ \ {\rm and}\ \ l^{-1}\gamma l\in p^{-1}(\overline{K\cap P})\}\\
&=\{lnk| \ l\in A_{M}(\gamma)\setminus \Omega(K,P), n\in N, k\in K \ \ {\rm and} \ \ n^{-1}l^{-1}\gamma ln \in K\}.
\end{align*}
By Lemma \ref{lemmaNgammacompact} applied to $\Omega =\Omega(K,P)$, the set $\Omega (K)$ is compact. 


\end{proof}

\section{Proof of Theorem \ref{thmmain}}
The second main ingredient of the proof of Theorem \ref{thmmain} is an upper bound on the irreducible characters of compact $p$-adic groups \cite[Theorem 5.8.]{Fraczyk}.
\newline 
If $U$ is a compact subgroup of $\G(F)$ and $\rho$ is an irreducible representation of $U$ then we denote by $\chi _{\rho}$ the character of $\rho$ extended to $\G(F)$ by $0$.
\begin{proposition}{\cite[Theorem 5.8.]{Fraczyk}}\label{propositionFraczyk}
Let $\H$ be a connected reductive group over a non-Archimedean local field $F$.  Denote by $W$ the Weyl group of $\H$. For every compact subgroup $U\subseteq \H(F)$, every irreducible (complex) representation $\rho$ of $U$ we have
\begin{equation*} 
|\chi _{\rho} (\gamma)|\leqslant |\Delta (\gamma)|_{F}^{-1} 2^{\dim \H- {\rm rk} (\H)}|W|
\end{equation*}
for every regular semisimple element $\gamma \in \H(F)$.
\end{proposition} 
For readers' convenience we prove the following, rather standard, fact.
\begin{lemma} \label{lemmachirho}
The $\chi _{\rho}$ is a finite sum of matrix coefficients of $\pi$.
\end{lemma} 
\begin{proof} 
Denote by $W$ the vector space associated to $\rho$. The trace of $\rho$ is equal to 
\begin{equation*} 
\sum _{i} \langle w_{i},\rho(g)w_i \rangle
\end{equation*} 
for any orthonormal basis $\{w_{i}\}$ of $W$. Define $f_{w_{i}}\in \cInd _{K}^{G}\rho$ as follows
\begin{equation*} 
f_{w_i}(g)=
\begin{cases}
\rho(g)w_{i}& {\rm if}\ \  g\in K\\
 0& {\rm otherwise}.
\end{cases}
\end{equation*} 
Write $\langle \cdot,\cdot \rangle $ 
for the inner product on $\cInd _{K}^{G}\rho$ given by $\langle f,g\rangle =\int_{G}\langle f(h), g(h)\rangle dh$ for the Haar measure $dh$ which is $1$ on $K$. 
We have \begin{equation*} 
\chi_{\rho}(g) =\sum _{i} \langle f_{w_{i}},\pi (g) f_{w_{i}}\rangle 
\end{equation*} 
for any $g\in G $. 
\end{proof}

\begin{proof}[Proof of Theorem \ref{thmmain}]
Let $\pi$ be an irreducible supercuspidal representation of $G$ of the form $\cInd _{J}^G \Lambda $ with an open compact modulo center $J$. Since $\G$ is semisimple we can assume $\pi=\cInd _{K}^{G}\rho $ for a maximal compact open subgroup $K$ of $G$ and an irreducible representation $\rho$ of $K$.

To estimate $\theta_{\pi}(\gamma)$ we use Arthur's formula \cite[Theorem]{Arthur1987}. 
Let $P$ be a parabolic subgroup of $\G(F)$ which contains $\gamma $ and minimal with this property. Let $M$ be the Levi component of $P$ containing $\gamma $.
The element $\gamma $ is $M$-elliptic over $F$, i.e. the centralizer of $\gamma $ in $M$ is compact modulo $A_M(\gamma)$. 
\newline Recall now the definition of the weight function $v_{M}$ appearing in the Arthur's formula. Choose a special maximal compact subgroup $K_{1}$ of $\G(F)$ which is in a good position relative to $M$ (i.e. the vertex of $K_{1}$ in the building of $\G$ belongs to the apartment of a maximal split torus of $M$). Denote by $\mathcal{P}(M)$ the set of all parabolic subgroups of $\G$ defined over $F$ with Levi component $M$. Let $P'\in \mathcal{P}(M)$ be a parabolic subgroup with Levi decomposition $P'=MN_{P'}$. Any point $x\in \G(F)$ can be written in the form
\begin{equation*} 
x=n_{P'}(x)m_{P'}(x)k_{P'}(x),
\end{equation*}
where $n_{P'}(x)\in N_{P'}(F)$, $m_{P'}(x)\in M(F)$ and $k_{P'}(x)\in K_{1}$. Write $X(M)$ for the module of rational characters of $M$. Consider a map  
\begin{equation*} 
H_{M}\colon M(F) \to \mathbf{a}_{M}:= \Hom (X(M)_{F}, \R)
\end{equation*} 
given by $\exp{ (\langle H_{M}(m),\chi \rangle )}=|\chi (m)|$ for $m\in M(F)$, $\chi \in X(M)_{F}$. We put 
\begin{equation*} 
H_{P'}(x):=H_{M}(m_{P'}(x)).
\end{equation*}
The function $v_{M}$ is defined as the volume of the convex hull of the projection of 
\begin{equation*} 
\{ -H_{P'}(x)| \ \ P'\in \mathcal{P}(M)\}
\end{equation*} onto $\mathbf{a}_{M}/\mathbf{a}_{G}$. Remark that by definition, the function $v_{M}$ is constant on right $K_{1}$-cosets, so in particular it is continuous. 
\newline 

By Lemma \ref{lemmachirho} and \cite{Arthur1987}, 
\begin{equation}\label{equationArthur} 
\int_{A_{M}(\gamma)\setminus G}f(x^{-1}\gamma x)v_{M}(x)d\mu (x)=(-1)^{\dim (A_{M}(\gamma)/ A_{\G})}\theta_{\pi}(f)\theta_{\pi}(\gamma)
\end{equation}
for $f=\chi _{\rho}$.
 By Proposition \ref{propositionomegacompact}, $\Omega (K)$ has finitely many right cosets modulo $K$. By Frobenius reciprocity, 
\begin{equation*} 
\theta_{\pi}(\chi_{\rho})=\tr\,\pi (\chi_{\rho})=\dim \Hom_{K}(\rho,\pi)=\dim \Hom_{G}(\cInd_{K}^{G}\rho , \pi)=\dim \Hom_{G}(\pi , \pi )=1.
\end{equation*}

By (\ref{equationArthur}), 
\begin{align*} 
|\theta _{\pi }(\gamma)| &=\left| \int _{A_{M}(\gamma)\setminus G}\chi _{\rho} (g^{-1}\gamma g) v_{M}(g)dg\right|=\left| \int _{\Omega (K)}\chi _{\rho}(g^{-1}\gamma g) v_{M}(g)dg \right|\\
&\leqslant \sum _{gK\in \Omega (K)/ K}\int _{K}\left|\chi _{\rho}(k^{-1}g^{-1}\gamma gk)\right||v_{M}(gk)|dk. 
\end{align*} 
Since the function $v_{M}$ is continuous and $\Omega (K)$ is compact, there exists a constant $C>0$ dependent on $\gamma $ such that $|v_{M}(gk)|< C$ for every $k\in K$, $g\in \Omega (K)$. Together with Proposition \ref{propositionFraczyk} it implies
\begin{align*} 
|\theta _{\pi}(\gamma)|&\leqslant \left( \sum _{g\in\Omega(K)/K} |\Delta (g^{-1}\gamma g )|_{F}^{-1}
C\right)2^{\dim \G -{\rm rk} (\G)}|W|\\
&= |\Delta (\gamma)|_{F}^{-1}|\Omega(K)/K|C2^{\dim \G - \textrm{rk} (\G )} |W|=: \kappa _{\gamma}
\end{align*} 
for some $\kappa _{\gamma}$ dependent only on $G$ and $\gamma$. We can take $\kappa _{\gamma} $ to be independent of $K$ because there is only finitely many $G$-conjugacy classes of maximal compact subgroups of $G$. 
\end{proof} 

\section{Explicit bound for $\gamma$ elliptic}\label{ellipticexplicit}
In this section we compute an explicit bound on $\theta _{\pi}(\gamma)$ for $\gamma $ elliptic. 

\begin{theorem}\label{thmellipticexplicit} 
Let $\G$ be a connected semisimple group over a non-Archimedean local field $F$ of characteristic zero. Let $W$ be the Weyl group of $\G $. Then, for every regular elliptic $\gamma \in \G(F)$ and every irreducible compactly induced supercuspidal representation $\pi$ we have 
\begin{equation*} 
|\theta _{\pi}(\gamma)|\leqslant |\Delta(\gamma)|_{F}^{-1}2^{\dim \G-{\rm rk }\,\G}|W|.
\end{equation*} 
\end{theorem} 
 
\begin{corollary} 
Let $\G$ be a connected semisimple group over a non-Archimedean local field $F$ of characteristic zero. Let $W$ be the Weyl group of $\G$. Then, for every regular elliptic element $\gamma \in \G(F)$ and every irreducible compactly induced supercuspidal representation $\pi$ with $\deg (\pi)\geqslant |\Delta (\gamma )|_{F}^{-1}$ we have 
\begin{equation*} 
|\theta _{\pi} (\gamma) |\leqslant |\Delta (\gamma)| _{F}^{-\frac{1}{2}}\deg(\pi)^{\frac{1}{2}}2^{\dim \G-{\rm rk}\, \G }|W|.  
\end{equation*} 
\end{corollary} 
This goes towards \cite[Conjecture 3.15. (iii)]{KST2}.
\\The proof of Theorem \ref{thmellipticexplicit} is inspired by the proof of \cite[Theorem 5.8.]{Fraczyk}. However, for $\G$ non-compact some significant changes are needed. 
In order to estimate $\theta _{\pi}$, we rely on Kazhdan's orthogonality relations (Lemma \ref{lemmaorth}) and the following lemma. 

Fix a regular elliptic element $\gamma \in \G(F)$. Let $\mathbf{T}_{0}$ be an elliptic torus such that $\gamma \in \mathbf{T}_{0}(F)$. Denote by $\G^{rs}(F)$ the set of regular semisimple elements in $\G(F)$. For any irreducible representation $\tau $ of a compact group we write $\chi _{\tau}$ for its trace. 

\begin{lemma}\label{lemmasumofcharacters} 
Let $S$ be a closed subset of $\mathbf{T}_{0}(F)\cap \G^{rs}(F)$. 
Then, $\theta _{\pi}|_{S}$ is the restriction of a finite integral combination of unitary characters of $\mathbf{T}_{0}(F)$. 
\end{lemma} 
\begin{proof} 
For any regular elliptic $g\in \mathbf{G} (F)$ write 
\begin{equation*} 
\Omega _{g}(K):= \{ x\in \G(F):\ \ x^{-1}gx\in K\}.  
\end{equation*}
By Proposition \ref{propositionomegacompact}, Arthur's formula and the local constancy of $\theta _{\pi}$, for every regular elliptic $g$ there exists an open neighborhood $U_{g}$ of $g$ such that 
\begin{equation*} 
\theta _{\pi}(u_{g})=\sum _{xK\in\Omega_{g}(K)/K}\chi_{\rho}(x^{-1}u_{g}x)  
\end{equation*} for every $u_{g}\in U_{g}$. Choose a finite cover $\{ U_{g_{i}}\} _{i=1}^{n}$ of $S$ with $g_{i}\in S$. 
Let $\mathcal{F}:=\bigcup _{i=1}^{n}\Omega _{g_{i}}(K)$. This set is left $\mathbf{T}_{0}(F)$-invariant. Consider the representation 
\begin{equation*} 
\sigma := \bigoplus _{x\in \mathbf{T}_{0}(F)\setminus \mathcal{F} /K} {\rm Ind}_{\mathbf{T}_{0}(F)\cap K^{x^{-1}}}^{\mathbf{T}_{0}(F)}\rho ^{x^{-1}} 
\end{equation*} 
where $K^{x^{-1}}:=xKx^{-1}$ and $\rho ^{x^{-1}}(a):=\rho (x^{-1}ax)$ for any $a\in K^{x^{-1}}$. For every $g_{i}$ we have 
\begin{align*} 
\chi_{\sigma}(g_{i})&=\sum _{x\in \mathbf{T}_{0}(F) \setminus \mathcal{F} / K}\ \ \sum _{\substack{ r\in \mathbf{T}_{0}(F)/ \mathbf{T} _{0} (F) \cap K^{x^{-1}},\\ g_{i} \in K^{x^{-1}}}}  \chi _{\rho^{x^{-1}}}(g_{i})=\sum _{xK\in \Omega _{g_{i}}(K)/K}\chi _{\rho}(x^{-1}g_{i}x). 
\end{align*}

Therefore $\theta _{\pi}|_{S}=\chi _{\sigma}|_{S}$. \end{proof} 
\begin{lemma}\label{lemmaorth} 
Let $\G$ be a semisimple group over a non-Archimedean local field $F$. For any torus $\mathbf{T}$ write $W_{\mathbf{T}}:=N(\mathbf{T}(F))/\mathbf{T}(F)$ where $N(\mathbf{T}(F))$ denotes the normalizer of $\mathbf{T}(F)$ in $\G(F)$. Then, for every supercuspidal representation $\pi$ of $\G(F)$ we have 
\begin{equation*} 
\sum _{[\mathbf{T}]\subseteq \G}\frac{1}{|W_{\mathbf{T}}|}\int _{\mathbf{T}(F)\cap \G^{rs}(F)}|\Delta (t)|_{F}|\theta _{\pi} (t)|^{2}dt=1 
\end{equation*} 
where $[\mathbf{T}]$ runs over all $\G(F)$-conjugacy classes of maximal elliptic tori and $dt$ denotes the probabilistic Haar measure on $\mathbf{T}(F)$.
\end{lemma}\begin{proof} Follows from the Kazhdan's orthogonality relations \cite[Theorem III.4.21]{SchneiderStuhler} and \cite[Theorem III.4.6.]{SchneiderStuhler}. We apply \cite[\S 3, Lemma 1.]{Kazhdan} and Weyl integration formula to determine the measure $dc$ in \cite[Theorem III.4.21]{SchneiderStuhler}. 
\end{proof} 
\begin{proof}[Proof of Theorem \ref{thmellipticexplicit}] Denote by $\mathbf{T}_{0}$ maximal split torus such that $\gamma \in \mathbf{T}_{0}(F)$. 
Let $m:=\frac{1}{2}(\dim \G- {\rm rk}\, \G )$ and choose a set of positive roots $\lambda _{1},\ldots , \lambda _{m}$ of $\mathbf{T}_{0}$ defined over the algebraic closure of $F$. We have 
\begin{equation*} 
\Delta (t)=\prod _{j=1}^{m}(1-\lambda _{j}(t))(1-\lambda _{j}(t)^{-1}). 
\end{equation*}Let 
\begin{equation*} 
Z_{i}:=\{ t\in \mathbf{T}_{0}(F)|\ \ |1-\lambda _{i}(t)|_{F}<|1-\lambda _{i} (\gamma)|_{F}\}. 
\end{equation*} 
Let $E$ be a finite extension of $F$ such that every character of $\mathbf{T}_{0}$ is defined over $E$. Denote by $\mathcal{O}_{E}^{\times}$ the group of all invertible elements in the ring  of integers of $E$. By the compactness of $\mathbf{T}_{0}(F)$, $\lambda _{j}(\mathbf{T}_{0}(F))\subseteq \mathcal{O}_{E}^{\times}$ for every $j\in\{1,\ldots, m\}$. By Pontryagin duality, for every $j\in\{1, \ldots , m\}$ we can find a unitary character $\alpha _{j}\colon \mathcal{O}_{E}^{\times} \to \mathbb{C}^{\times}$ such that $\alpha _{j}(\lambda _{j}(\gamma ))\neq 1$ and $\alpha _{j}(\lambda _{j}(z))=1$ for every $z\in Z_{j}$. Define 
\begin{equation*} 
\beta (t):=\prod _{j=1}^{m}(1-\alpha _{j}(\lambda _{j}(t))) 
\end{equation*} 
for every $t\in \mathbf{T}_{0}(F)$. Since $\beta (t)$ is a non-zero algebraic integer, there exists a Galois automorphism $\kappa$ such that 
\begin{equation*} 
\alpha (t):=\prod _{j=1}^{m}(1-\kappa(\alpha_{j}(\lambda _{j}(t))))
 \end{equation*} 
satisfies $|\alpha (t)|\geqslant 1$. 
Since $\alpha (z)=0$ for every $z\in \bigcup _{j=1}^{m}Z_{j}$, for every $t\in \mathbf{T}_{0}(F)$ we have
\begin{equation}\label{equationdelta} 
\frac{|\alpha (t)|^{2}}{|\Delta (t)|_{F}} \leqslant \frac{|\alpha (t) |^{2}}{|\Delta (\gamma )|_{F}}.  
\end{equation}
Let $dt$ be the probabilistic Haar measure on $\mathbf{T}_{0}(F)$. By Lemma \ref{lemmaorth},
 \begin{equation} \label{equationorth} 
 \int _{\mathbf{T}_{0}(F)\cap \G^{rs}(F)}|\Delta (t)|_{F}|\theta _{\pi}(t)|^{2}dt\leqslant |W|. 
 \end{equation}Putting (\ref{equationdelta}) and (\ref{equationorth}) we get 
 \begin{equation}\label{equationl} 
 \int _{\mathbf{T}_{0}(F)\cap \G^{rs}(F)}|\alpha (t)|^{2}|\theta _{\pi}(t)|^{2}dt\leqslant \sup _{t\in \mathbf{T}_{0}(F)\cap \G^{rs}(F)}|\alpha (t)|^{2}|\Delta (\gamma)|_{F}^{-1}|W|\leqslant 2^{2m}|\Delta (\gamma )|_{F}^{-1}|W|. 
 \end{equation} 
 Define $R:=\{ t\in \mathbf{T}_{0}(F)\cap \G^{rs}(F)| \ \ \alpha (t)\neq 0\}$. By Lemma \ref{lemmasumofcharacters} applied to $R$, $\theta _{\pi}$ is a finite integral combination of unitary characters so 
 \begin{equation*} 
 \alpha\theta _{\pi }|_{R}=\sum _{\varphi } a_{\varphi }\varphi,  
 \end{equation*} 
 where $a_{\varphi}\in \mathbb{Z}$ and each $\varphi$ is a unitary character of $\mathbf{T}_{0}(F)$. By (\ref{equationl}), \begin{equation*}\sum _{\varphi } a _{\varphi }^{2} =\int _{\mathbf{T}_{0}(F)\cap  \mathbf{G} ^{rs} (F)} |\alpha (t)|^{2} |\theta _{\pi } (t)|^{2} dt\leqslant 2^{2m}|\Delta (\gamma)|_{F}^{-1}|W|.\end{equation*} Since $a _{\varphi } \in \mathbb{Z}$, 
 \begin{equation*} 
 |\theta _{\pi}(\gamma)|\leqslant |\alpha (\gamma )\theta _{\pi}(\gamma)|=\left|\sum _{\varphi} a_{\varphi}\varphi (\gamma )\right| \leqslant \sum _{\varphi} |a_{\varphi}| \leqslant 2 ^{\dim \G- {\rm rk}\, \G}|\Delta (\gamma )|_{F}^{-1}|W|. 
 \end{equation*} 
 \end{proof} 

\bibliography{ref}
\bibliographystyle{plain}
\end{document}